\documentclass[a4paper,10pt]{article}
\usepackage[utf8x]{inputenc}
\usepackage{amsmath,amssymb, amsthm}
\usepackage{enumerate}
\usepackage{verbatim}

\newcommand{\Z}{\mathbb{Z}}

\newcommand{\R}{\mathbb{R}}
\newcommand{\C}{\mathbb{C}}

\newcommand{\OreGen}{R[x;\sigma,\delta]}

\newcommand{\DiffPol}{R[x;\identity_R,\delta]}

\DeclareMathOperator{\identity}{id}
\DeclareMathOperator{\Cen}{Cen}

\makeatletter
\def\imod#1{\allowbreak\mkern10mu({\operator@font mod}\,\,#1)}
\makeatother

\theoremstyle{plain}
\newtheorem{theorem}{Theorem}[section]
\newtheorem{lemma}[theorem]{Lemma}
\newtheorem{remark}[theorem]{Remark}

\newtheorem{corollary}[theorem]{Corollary}
\newtheorem{definition}[theorem]{Definition}

\title{Centralizers in non-associative rings with a pseudo-degree function }
\author{Johan Richter %
\footnote{Blekinge Institute of Technology, Karlskrona, Sweden	 E-mail: johan.richter@bth.se }
}

\begin{document}

\maketitle

\begin{abstract}
This papers studies centralizers of an element, $a$, in the nucleus of a non-associative algebra with a special type of valuation. We prove that the centralizer of $a$ is a free module of finite rank over the algebra generated by $a$. 

\end{abstract}

\section{Introduction}

 Burchnall and Chaundy studied, in a series of papers in the 1920s and 30s \cite{BurchnallChaundy1,BurchnallChaundy2, BurchnallChaundy3}, the properties of commuting pairs of ordinary differential operators. The following theorem is essentially found in their papers. 
\begin{theorem}\label{thm_BC}
 Let $P=\sum_{i=0}^n p_i D^i$ and $Q= \sum_{j=0}^{m} q_j D^j$ be two commuting differential operators, with polynomial coefficients and constant leading coefficients. Then there is a non-zero polynomial $f(s,t)$ in two commuting variables over $\C$ such that 
$f(P,Q) =0$. Note that the fact that $P$ and $Q$ commute 
guarantees that $f(P,Q)$ is well-defined.
\end{theorem} 

The result of Burchnall and Chaundy was rediscovered independently during the 70s by researchers in the area of PDEs. It turns out that several important equations can be equivalently 
formulated as a condition that a pair of differential operators commute. These differential equations are completely integrable as a result, which roughly means that they possess an infinite number of conservation
laws. In fact  Theorem~\ref{thm_BC} was rediscovered by Kricherver \cite{Krichever} as part of his research into integrable systems. 

To state some generalizations of Burchnall's and Chaundy's result  we shall recall a definition. 

\begin{definition}

Let $R$ be a ring, $\sigma$ an endomorphism of $R$ and $\delta$ an additive function, $R \to R$, satisfying 
\begin{equation*}
\delta(ab) = \sigma(a)\delta(b)+\delta(a)b
\end{equation*}
for all $a,b \in R$. (Such $\delta$:s are known as $\sigma$-derivations.) The \emph{Ore extension} $\OreGen$ is the polynomial ring $R[x]$ equipped with a new multiplication such that $xr=\sigma(r)x+\delta(r)$ for all $r \in R$. Every element of $\OreGen$ can be written uniquely as  $\sum_i a_i x^i$ for some $a_i \in R$.

If $\sigma=\identity$ then $\DiffPol$ is called a \emph{differential operator ring}. If $P= \sum_{i=0}^n a_i x^i$, with $a_n \neq 0$, we say that $P$ has \emph{degree} $n$. We say that the zero element has degree $-\infty$. 
\end{definition}

The ring of differential operators studied by Burchnall and Chaundy can be taken to be the Ore extension $T=C^{\infty}(\R,\C)[D;\identity,\delta]$, where $\delta$ is the ordinary derivation.

In a paper by Amitsur \cite{Amitsur} one can find the following theorem.

\begin{theorem}\label{thm_DiffField}
 Let $K$ be a field of characteristic zero with a derivation $\delta$. Let $F$ denote the subfield of constants. (By a \emph{constant} we mean an element that is mapped to zero by the derivation.) Form the differential operator ring $S=K[x; \identity,\delta]$, and let $P$ be an element of $S$ of 
degree $n>0$. Set $F[P]= \{ \sum_{j=0}^m b_j P^j \ | \ b_j \in F \ \}$, the ring of polynomials in $P$ with constant coefficients. 
Then the centralizer of $P$ is a commutative subring of
$S$ and a free $F[P]$-module of rank at most $n$.  
\end{theorem}

Later authors have found other contexts where Amitsur's method of proof can be made to work. We mention an article by Goodearl and Carlson \cite{GoodearlCarlson}, and one by Goodearl alone \cite{GoodearlPseudo}, that generalize Amitsur's result to a wider class of rings. The proof has also been generalized by Bavula \cite{Bavula}, Mazorchouk \cite{Mazorchuk} and Tang \cite{Xin}, among other authors. As a corollary of these results, one can recover Theorem \ref{thm_BC}.

Among papers inspired by Amitsur's is a paper by Hellstr\"{o}m and Silvestrov \cite{ergodipotent}. Hellstr\"{o}m and Silvestrov study graded algebras satisfying a condition they call $l$-BDHC (short for ``Bounded-Dimension Homogeneous Centralizers''). 

\begin{definition}
Let $K$ be a field, $\ell$ a positive integer and $S$ a $\Z$-graded $K$-algebra. The homogeneous components of the gradation are denoted $S_m$, for $m \in \Z$. Let $\Cen(n,a)$, for $n \in \Z$ and $a \in S$, denote the elements in $S_n$ that commute with $a$. We say that $S$ has $\ell$-BDHC if for all $n \in \Z$, nonzero $m \in \Z$ and nonzero $a \in S_m$, it holds that $\dim_K \Cen(n,a) \leq \ell$.  

\end{definition}

Hellstr\"{o}m and Silvestrov apply the ideas of Amitsur's proof. They need to modify them however, especially to handle the case when $\ell>1$. 

To explain their results further, we introduce some more of their notation. Denote by $\pi_n$ the projection, defined in the obvious way, from $S$ to $S_n$. Hellstr\"{o}m and Silvestrov define a function $\bar{\chi}: A\setminus \{ 0\} \to \Z$ by
\begin{equation*}
\bar{\chi}(a) = \max\{ \, n \, \in \Z \, | \, \pi_n(a) \neq 0 \, \},
\end{equation*}
and set $\bar{\chi}(0) = -\infty$. 
Set further $\bar{\pi}(a) = \pi_{\bar{\chi}(a)}(a)$. 

Now we have introduced enough notation to state the relevant results. The following result is the main part of Lemma 2.4 in their paper. 

\begin{theorem}\label{thm_HS1}
Assume $S$ is a $K$-algebra with $l$-BDHC and that there are no zero divisors in $S$. If $a \in S\setminus S_0$ is such that $\bar{\chi}(a)=m >0$ and $\bar{\pi}(a)$ is not invertible in $S$, then there exists a finite $K[a]$-module basis $\{b_1,\ldots, b_k\}$ for the centralizer of $a$. Furthermore $k \leq ml$. 
\end{theorem}

The reason they refer to it as a lemma is that their main interest is in the following corollary of this result. (Which is proved the same way as Corollary \ref{corAlgDep} in this paper.) 

\begin{theorem}\label{thm_HS2}
Let $K$ be a field and assume the $K$-algebra $S$ has $l$-BDHC and that there are no zero divisors in $S$. If $a \in S\setminus S_0$ and $b\in S$ are such that $ab = ba$, $\bar{\chi}(a)>0$ and $\bar{\pi}(a)$ is not invertible in $S$, then there exists a nonzero polynomial $P$ in two commuting variables with coefficients from $K$ such that $P(a,b)=0$. 
\end{theorem}

Theorem \ref{thm_HS2} is directly analogous to Theorem \ref{thm_BC}.

The proofs in \cite{ergodipotent}  are based upon certain properties of the function $\bar{\chi}$ they define.  In \cite{Pseudodegree} the author of the present article uses what he calls \emph{pseudo-degree functions} to axiomatize the properties Hellstr\"om and Silvestrov need, and generalize their result. He manages to prove Theorem \ref{thm_BoundDim} from the present article in the associative case, as well as number of corollaries.

It shall be the goal of this paper to show that some of the results of \cite{Pseudodegree} also can be proven in a non-associative setting, for elements in the nucleus of the algebra.

\subsection{Notation and conventions}

By $K$ we will always denote a field.

If $R$ is a ring then $R[x_1,x_2, \ldots x_n]$ denotes the ring of polynomials over $R$ in central indeterminates $x_1,x_2,\ldots,x_n$. 

All rings and algebras are assumed to be unital, but not necessarily associative. We shall work only with algebras over fields. If $S$ is a $K$-algebra, where $K$ is a field, we embed $K$ as a subset of $S$ in the natural way. 

If $S$ is an algebra, by $N_l(S)$ we mean the set of elements in $S$ that associate with everything from the left, that is the set of all $a$ such that $a(bc)=(ab)c$ for all $b,c\in S$. Similarly we define $N_m(S)$ and $N_r(S)$. These subsets are in fact subalgebras of $S$. We define the nucleus, $N(S)$, as the intersection of $N_l(S)$, $N_m(S)$ and $N_r(S)$.  By $C(S)$ we mean the set of element that commute with everything in $S$. The \emph{center}, $Z(S)$, is defined as $C(S)\cap N(S)$. 

Obviously, $K \subseteq Z(S)$. 

If $S$ is a ring and $a$ is an element in $S$, the \emph{centralizer} of $a$, denoted $C_S(a)$, is the set of all elements in $S$ that commute with $a$.

\section{Centralizers in algebras with degree functions}
 We start by defining the properties our pseudo-degree functions will need in the non-associative setting. 
 
\begin{definition}
Let $K$ be a field and let $S$ be a $K$-algebra. A function, $\chi$, from $S$ to $\Z \cup \{-\infty\}$ is called a pseudo-degree function if it satisfies the following conditions: 

\begin{itemize}
\item $\chi(a) = -\infty$ iff $a=0$,
\item $\chi(ab) = \chi(a)+\chi(b)$ for all $a,b \in S$, 
\item $\chi(a+b) \leq \max(\chi(a),\chi(b))$,
\end{itemize}

\end{definition}

This is essentially a special case of the concept of a \emph{valuation}. 

The following Lemma is obvious. 

\begin{lemma}
Let $K$ be a field and let $S$ be a $K$-algebra. If $\chi$ is a pseudo-degree function on $S$ then 
\begin{itemize}
\item 	$\chi(1)=\chi(-1)=0$
\item  if $b \in S^{\times} \cap N_m(S) \cap N_r(S)$ then $\chi(b^{-1}) = -\chi(b)$
\end{itemize}
\end{lemma}

	We also need a condition that can replace $l$-BDHC. We formulate it next.

\begin{definition}

Let $K$ be a field, with characteristic not equal to $2$, $S$ a $K$-algebra with a pseudo-degree function, $\chi$, and $\ell$ be a positive integer. A subalgebra, $B \subset A$, is said to satisfy condition $D(\ell)$ if  $\chi(b) \geq 0$ for all non-zero $b \in B$ and if, whenever we have $\ell+1$ elements $b_1, \ldots, b_{l+1} \in B$, all mapped to the same integer by $\chi$, there exist $\alpha_1,\ldots, \alpha_{\ell+1} \in K$, not all zero, such that $\chi\left(\sum_{i=1}^{l+1} \alpha_i b_i\right) < \chi(b_1)$. 

\end{definition}

\begin{remark}
Note that the requirement that $\alpha_1,\ldots,\alpha_{l+1}$ are mapped to the same integer by $\chi$ excludes the possibility that they are equal to $0$.   

\end{remark}

\begin{remark}\label{remScalarDeg}

Suppose that $S$ is a $K$-algebra and $a \in N_l(S)\cap N_m(S)$ is such that $C_S(a)$ satisfies condition $D(\ell)$ for some $\ell$. If $\alpha$ is a non-zero scalar, then $\chi(\alpha^{-1})=-\chi(\alpha)$. Since $\chi(\alpha)$ and $\chi(\alpha^{-1})$ are both non-negative, it follows that $\chi(\alpha)=0$. 
\end{remark}

\begin{lemma}\label{lemAdd}

Let $K$ be a field. Suppose that $S$ is an $K$-algebra and $\chi$ is a pseudo-degree function on $S$ that maps all the non-zero scalars to zero. Then if $a,b \in S$  are such that $\chi(b)<\chi(a)$, the identity 
\begin{equation}
\chi(a+b) =\chi(a)
\end{equation}
holds.
\end{lemma}

\begin{proof}
On the one hand we find $\chi(a+b) \leq \max(\chi(a),\chi(b))=\chi(a)$. On the other hand 
$\chi(a) = \chi(a+b-b) \leq \max(\chi(a+b),\chi(b))$ Since $\chi(b)<\chi(a)$ we must have 
$\chi(a) \leq \chi(a+b)$. 
\end{proof}

We now proceed to prove an analogue of Theorem \ref{thm_HS1}, using just the existence of some pseudo-degree function and the condition $D(\ell)$.

\begin{remark}
If $a\in N_l(S) \cup N_m(S)\cup N_r(S)$ then the algebra generated by $a$ is associative.  
\end{remark}

\begin{theorem}\label{thm_BoundDim}
Let $K$ be a field and let $S$ be a $K$-algebra. Suppose $S$ has a pseudo-degree function, $\chi$.

Let $a \in N(S)$, with $m=\chi(a)>0$, such that $C_S(a)$ satisfies condition $D(\ell)$ for some positive integer $\ell$. The algebra generated by $a$, which we write $K[a]$, is isomorphic to a polynomial ring in one variable over $K$.  Also $C_S(a)$ is a free left $K[a]$-module of rank at most $\ell m$.

\end{theorem}

\begin{proof}
	
Note first that $C_S(a)$ is an algebra if $a\in N(S)$.   
	
Since $a\in N(S)$, $K[a]$ is an associative and commutative algebra. Since $\chi(a)>0$ it follows that $K[a]$ is isomorphic to a polynomial ring. It is easy to check that $C_S(a)$ is closed under multiplication from the left by elements in $K[a]$. 	

Construct a sequence $b_1,b_2,\ldots $ by setting $b_1=1$ and choosing $b_{k+1} \in C_S(a)$ such that $\chi(b_{k+1})$ is minimal subject to the restriction that $b_{k+1}$ does not lie in the $K[a]$-linear span of $\{b_1, \ldots, b_k\}$. We will show later in the proof that such a sequence has at most $lm$ elements. 

We first claim that 
\begin{equation}\label{degreeSum}
\chi \left(\sum_{i=1}^k \phi_i b_i\right) = \max_{i\leq k} (\chi(\phi_i) +\chi(b_i)),
\end{equation} 
for any $\phi_1,\ldots \phi_k \in K[a]$. We show this by induction on $n=\max_{i\leq k} (\chi(\phi_i) +\chi(b_i))$. It is clear that the left-hand side of \eqref{degreeSum} is never greater than the right-hand side. When $n=-\infty$ Equation \eqref{degreeSum} obviously holds. If $n=0$, Equation \eqref{degreeSum} holds since $\chi(b)\geq 0$ for all non-zero $b \in C_S(a)$, and $\sum_{i=1}^k \phi_i b_i=0$ would imply that all $\phi_i=0$ by the choice of the $b_i$. That $\chi(b) \geq 0$ for all non-zero $b$ in $C_S(a)$ also means that no value of $n$ between $-\infty$ and $0$ is possible. 

For the induction step, assume \eqref{degreeSum} holds when the right-hand side is strictly less than $n$. To verify that it holds for $n$ as well, we can assume without loss of generality  that $\chi(\phi_k) + \chi(b_k)=n$, since if $\chi(\phi_j b_j)<n$ for some term $\phi_j b_j$ we can drop it without affecting either side of \eqref{degreeSum}, by Lemma \ref{lemAdd}. If $\phi_k \in K$ then $\chi(\phi_k)=0$, by Remark \ref{remScalarDeg}, and thus $\chi(b_k)=n$. By the choice of $b_k$ it then follows that $\chi(\sum_{i=1}^k \phi_i b_i) \geq n$, as otherwise $\sum_{i=1}^k \phi_i b_i$ would have been picked instead of $b_k$. If $\phi_k \notin K$, then $\chi(b_k) <n$ and thus $\chi(b_i) < n$ for $i=1,\ldots k$. Let $r_1,\ldots,r_k \in K$ and $\xi_1,\ldots,\xi_k \in K[a]$ be such that 
$\phi_i = a\xi_i +r_i$ for $i=1,\ldots,k$. We have $\chi(\sum_{i=1}^k r_i b_i) <n$ and thus by Lemma \ref{lemAdd}, the fact that $a\in N(S)$ and the assumptions on $\chi$ we get
\begin{equation*}
\chi\left(\sum_{i=1}^k \phi_i b_i \right) = \chi\left(\sum_{i=1}^k a\xi_i b_i +\sum_{i=1}^k r_ib_i \right) = \chi\left(a \sum_{i=1}^k \xi_i b_i \right) = m+\chi\left(\sum_{i=1}^k \xi_i b_i \right).
\end{equation*} 
We also have that $\max_{i\leq k} (\chi(\phi_i) +\chi(b_i)) = m +\max_{i\leq k} (\chi(\xi_i) +\chi(b_i))$.  By the induction hypothesis 
\begin{equation*}
\chi\left(\sum_{i=1}^k \xi_i b_i \right) = \max_{i\leq k} (\chi(\xi_i) +\chi(b_i)),
\end{equation*}
which completes the induction step. 

We now show that if $\chi(b_i) = \chi(b_j)$ for some $i\leq j$ then $j-i<l.$ Suppose $b_{i} ,\ldots,b_{i+l} $ all are mapped to the same integer, $q$, by $\chi$. Then there exists $\alpha_i,\ldots,\alpha_{i+l} \in K$, not all zero, such that 
\begin{equation*}
\chi\left(\sum_{p=i}^{i+l} \alpha_p b_p \right) <q,
\end{equation*}
by condition $D(\ell)$. This contradicts \eqref{degreeSum}.

It remains only to show that the sequence $(b_i)$ contains at most $\ell m$ elements. We will prove that every residue class $\pmod{m}$ can only contain at most $\ell$ elements. Suppose to the contrary, that we had elements $c_1,\ldots, c_{\ell+1}$, belonging to the sequence $(b_i)$ and all satisfying that $ \chi(c_i) \equiv \nu \imod{m}$. Set $k=\max_{1\leq i \leq l+1} (\chi(c_i))$ and define $\gamma_i = a^{\frac{k-\chi(c_i)}{m}}.$ Then $\chi(\gamma_i c_i) = k$, for all $i \in \{1, \ldots, \ell+1\}$, which implies that there exists  $\alpha_1,\ldots,\alpha_{\ell+1} \in K$, such that 
\begin{equation*}
\chi\left( \sum_{p=1}^{\ell +1} \alpha_i (\gamma_i c_i) \right) < k.
\end{equation*}
But this once again contradicts \eqref{degreeSum}.  
 
\end{proof}

We can also prove a result on the algebraic dependence of pairs of commuting elements.

\begin{corollary}\label{corAlgDep}
Let $S$ be a $K$-algebra with a pseudo-degree function, $\chi$. Let $a \in N(S)$ be  such that $C_S(a)$ satisfies Condition $D(\ell)$ for some $\ell>0$. Let $b \in C_S(a) \cap N_l(S)$. Then there exists a nonzero polynomial $P(s,t)\in K[s,t]$ such that $P(a,b) =0$. (Note that $P(a,b)$ is well-defined .)
\end{corollary}

\begin{proof}

Since $b\in  N_l(S)$ it follows that $b^k $ is well-defined for every positive integer $k$. Since $C_S(a)$ has finite rank as a left $K[a]$-module the elements $b,b^2,\ldots$ can not all be linearly independent over $K[a]$. Thus there exists $f_1(x),\ldots, f_k(x) \in K[x]$ , not all zero,  such that $\sum_{i=0}^k f_i(a) b^i =0$. Then $P(s,t) = \sum_{i=0}^k f_i(s) t^i =0$ is a polynomial with the desired property. 

\end{proof}

When $C_S(a)$ satisfies condition $D(1)$ we can say a little bit more. 

\begin{theorem}\label{thmComCen}
Let $K$ be a field and suppose $S$ is a $K$-algebra. Let $S$ have a pseudo-degree function, $\chi$. If $a\in N(S)$ satisfies $\chi(a)=m>0$ and $C_S(a)$ satisfies condition $D(1)$ then $C_S(a)$ has a finite basis as a $K[a]$-module, the cardinality of which divides $m$.

\end{theorem}

\begin{proof}

By Theorem \ref{thm_BoundDim} it is clear that there is a subset $H$ of $\{1, \ldots, m\}$
and elements $(b_i)_{i \in H}$ such that the $b_i$ form a basis for $C_S(a)$. By the proof of Theorem \ref{thm_BoundDim} it is also clear that $\chi(b_i) \neq \chi(b_j)$ if $i \neq j$. Without loss of generality we can assume $\chi(b_i)=i$ for all $i  \in H$. We can map $H$ into $\Z_m$ in a natural way. Denote the image by $G$. We want to show $G$ is a subgroup, for which it is enough to show that it is closed under addition. 

Suppose $g,h \in G$. There exists $i,j \in H$, with $i \equiv g \imod{m}$ and $j \equiv h \imod{m}$. We can write $b_ib_j= \sum_{k\in H} \phi_k b_k$, for some $\{\sigma_k \}$. It follows that 
\begin{equation*}
g+h \equiv i+j = \chi(b_ib_j)  = \max (\chi(\phi_k)+\chi(b_k)) \equiv \chi(b_r ) =r \imod{m} 
\end{equation*}
 for some $r \in H$. 

Since $G$ is a subgroup of $\Z_m$ it is clear that the cardinality of $G$, which is also the cardinality of $H$, must divide $m$.

\end{proof}

\section*{Example}

	We start with some notation we will use in this section. Let $\mathbb{O}$ be the octonions. Set $R=\mathbb{O}[y]$, let $\sigma$ be an $\mathbb{R}$-algebra endomorphism of $R$ such that $s=\deg_y(\sigma(y))>1$ and let $\delta$ be an $\mathbb{R}$-linear map such that $\delta(rs)= \sigma(r)\delta(s)+\delta(r)s$ for all $r,s\in R$. One can then form the non-associative Ore extension $S=\OreGen$.

	\begin{lemma}
	 If $a \in S $ is an element of positive degree with real coefficients, then $a$ belongs to $N(S)$.
	\end{lemma}

	\begin{proof}
	It is clear that $\mathbb{R}\subseteq N(S)$. From \cite[Proposition 3.4]{NonAssocOre} it follows that $X\in N(S)$. So $a \in N(S)$, since $N(S)$ is a subring of $S$. 
	
		\end{proof}

	\begin{theorem}
	If $a \in S $ is an element of positive degree with real coefficients then $C_S(a)$ is a free $\mathbb{R}[a]$-module of finite rank.  
	\end{theorem}

	\begin{proof}

	 We shall apply Theorem \ref{thmComCen}. To do so we need a pseudo-degree function. 
	
	The notion of the degree of an element in $S$ with respect to $x$ is obviously well-defined. Denote the degree of an element $b$ by $\chi(b)$. It is easy to see that $\chi$ satisfies all the requirement to be a pseudo-degree function. We proceed to show that $C_S(a)$ satisfies condition $D(8)$. Certainly it is true that $\chi(b)\geq 0$ for all nonzero $b \in C_S(a)$. 
	
	Let $b$ be a nonzero element of $S$ that commutes with $a$, such that $\chi(b) =n$. Suppose $\chi(a) =m$. By equating the highest order coefficient of $ab$ and $ba$ we find that 
	\begin{equation}\label{eqHigCoe}
		a_m \sigma^m (b_n) = b_n \sigma^n (a_m), 
	\end{equation}
	where $a_m$ and $b_n$ denote the highest order coefficients of $a$ and $b$, respectively. (Recall that these are polynomials in $y$.) We equate the degree in $y$ of both sides of \eqref{eqHigCoe} and find that 
	\begin{equation*}
		\deg_y(a_m) +s^m \deg_y(b_n) = \deg_y(b_n) +s^n \deg_y (a_m),
	\end{equation*}
	which determines the degree of $b_n$ uniquely. It follows that the solutions of \eqref{eqHigCoe} belong a $\mathbb{R}$-subspace of $\mathbb{O}[y]$ that is  eight-dimensional. This in turn implies that condition $D(8)$ is fulfilled. 
	
	We have now verified all the hypothesis necessary to apply Theorem \ref{thm_BoundDim}. 

	\end{proof}

	\begin{lemma}
	Let $\delta$ be the usual derivative on $R= \mathbb{O}[y]$ and set $T=R[x; \identity, \delta]$. If $a$ is an element of $T$ of positive degree with respect to $x$ and with coefficents that belong to $\mathbb{R}[y]$ then $a\in N(T)$. 
	\end{lemma}

	\begin{proof}
	It is clear that $\mathbb{R} \subset N(T)$ and  $x\in N(T)$ by \cite[Proposition 3.4]{NonAssocOre}. If we can show that $y \in N(T)$ then it will follow that $a\in N(T)$, since $N(T)$ is a ring. We will use the facts that $y\in N(R)$ and that $x\in N(T)$ in our calculations. 
	
	We first show that $y\in N_l(T)$. It is enough to show that $(y,bx^m, cx^n)=0$ for arbitrary $b,c  \in R$. Since $x^n \in N_r(T)$ one easily sees that $(y,bx^m, cx^n)= (y,bx^m,c)x^n$, so it is enough to prove that $(y,bx^m,c)=0$. But 
	$$(y,bx^m,c)= (y,b,x^mc) = (y,b,\sum_{i=0}^m \delta^{m-i}(c)x^i)= \sum_{i=0}^m(y,b, \delta^{m-i}(c))x^i =0. $$
	
	Now we prove that $y \in N_m(T)$. Let $b,c$ be arbitrary elements of $R$. It is enough to prove that $(bx^m,y,cx^n)=0$, which as before means it is enough to prove that $(bx^m,y,c)=0$.

	We adopt the convention that $\delta^n=0$ if $n$ is negative. We see that 
	\begin{equation*}
	(bx^m)(yc)= \sum_{i=0}^m \binom{m}{i} b \delta^{m-i}(yc).
	\end{equation*}  
	
	And
	\begin{align*}
	&(bx^my)c= \left(\sum_{i=0}^m \binom{m}{i} b\delta^{m-i}(y)x^i\right)c = (byx^m+mbx^{m-1})c=\\
	&= \left(\sum_{i=0}^m \binom{m}{i} (by)\delta^{m-i}(c) x^i\right) +m \sum_{i=0}^{m-1}\binom{m-1}{i} b \delta^{m-1-i}(c)x^i= \\
	& =\sum_{i=0}^m \binom{m}{i}b (y\delta^{m-i}(c)+(m-i)\delta^{m-1-i}(c))x^i.
	\end{align*}
	By the generalized Leibniz rule it follows that $(bx^m)(yc)=(bx^my)c$. 
	
	Remains to show that $y\in N_r(T)$, which for which it is enough to prove that $(bx^m,cx^n,y)=0$ for arbitrary $b,c \in R$. Since $x^n \in N(T)$ it is true that $(bx^m,cx^n,y)= (bx^m,c,x^ny)$. And 
	\begin{equation*}
	(bx^m,c,x^ny)=(bx^m,c,yx^n)+ (bx^m,c,nx^{n-1})= (bx^m,c,y)x^n+(bx^m,c,n)x^{n-1} 
	\end{equation*}
	Since $(bx^m,c,n)=0$, it is enough to prove that $(bx^m,c,y)=0$. But since $x^m \in N(T)$ it follows that $(bx^m,c,y)=(b,x^mc,y)= (b,\sum_{i=0}^m \delta^{m-i}(c)x^i,y)$. So we can further reduce our problem to proving that $(b,dx^i, y)=0$ for arbitrary $b,d\in R$, and since $(b,dx^i,y)= (b,d,x^iy)= (b,d,y)x^i+(b,d,i)x^{i-1}$, this is clearly true.

	\end{proof}

	\begin{theorem}
 	Let $\delta$ be the usual derivative on $R= \mathbb{O}[y]$ and set $T=R[x; \identity, \delta]$. If $a$ is an element of $T$ of positive degree with respect to $x$ and with that belong to $\mathbb{R}[y]$ then $C_T(a)$ is a free $\mathbb{R}[a]$-module of finite rank. 
			
	\end{theorem}
	
	\begin{proof}
		
	We will apply Theorem \ref{thmComCen} with the degree w.r.t. $x$ as the pseudo-degree function. We know that $a\in N(T)$.

	 We want to show that $C_T(a)$ satisfies condition $D(8)$. Clearly $C_T(a)$ does not contain any nonzero elements of negative degree. 
	
	Let $b$ be a nonzero element of $T$ that commutes with $a$, such that $\chi(b)=n$. Suppose $\chi(a)=m$. Let $a_m$ and $b_n$ denote the highest order coefficients (w.r.t. $x$) of $a$ and $b$, respectively, and similarily for $a_{m-1}$ and $b_{n-1}$. By equating the second highest coefficient of $ab$ and $ba$, and using the fact that the coefficents of $a$ belong to $Z(R)$, we see that
	\begin{equation}\label{eq_SecCoeff}
			 ma_m\delta(b_n)= nb_n\delta(a_m).
	\end{equation}

	Now set $\deg_y(a_m)=\alpha$ and $\deg_y(b_m)=\beta$. Let the leading coefficient in $a_m$ be $r$ and the leading coefficient in $b_n$ be $s$. Then it follows from \eqref{eq_SecCoeff} that 
	\begin{align*}
		mr \beta s &= n s \alpha r \Leftrightarrow \\
		m\beta sr & = n\alpha sr \Leftrightarrow \\
		m \beta &= n\alpha		
	\end{align*}
	(Since $m,n, \alpha, \beta$ and $r$ all belong to $Z(\mathbb{O}$) and $rs\neq 0$ the preceding calculations are easily seen to be justified.)
	Since $m\neq 0$ the equation $m\beta = n\alpha$ determines $\beta$ uniquely. So all solutions of \eqref{eq_SecCoeff} belong to an eight-dimensional $\mathbb{R}$-subspace of $\mathbb{O}[y]$, and condition $D(8)$ is satisfied. 
	\end{proof}

\bibliographystyle{amsplain}

\bibliography{General}

\end{document}